\documentclass[12pt]{amsart}

\usepackage{cite}
\usepackage[hypertexnames=false, colorlinks, linkcolor=blue, allcolors=blue]{hyperref}
\hypersetup{nesting=true,debug=true,naturalnames=true}
\usepackage[margin=1in]{geometry}
\usepackage{graphicx}
\usepackage{enumitem}
\usepackage{thm-restate}
\usepackage{cleveref}
\usepackage{overpic}
\usepackage{booktabs}
\usepackage{tikz-cd}

\usepackage[colorinlistoftodos]{todonotes}

\newtheorem{theorem}{Theorem}[section]
\newtheorem{lemma}[theorem]{Lemma}

\newtheorem{definition}[theorem]{Definition}

\newtheorem{question}[theorem]{Question}

\theoremstyle{definition}

\begin{document}

\title{Equivariant topological slice disks and negative amphichiral knots} 

\author{Keegan Boyle}  
\address{Department of Mathematics, University of British Columbia, Canada} 
\email{kboyle@math.ubc.ca}

\author{Wenzhao Chen}
\address{Department of Mathematics, University of British Columbia, Canada}
\email{chenwzhao@math.ubc.ca}

\setcounter{section}{0}

\begin{abstract}
We show that any strongly negative amphichiral knot with a trivial Alexander polynomial is equivariantly topologically slice.
\end{abstract}
\maketitle
\section{Introduction}

In the 1980s, Freedman proved that any knot with a trivial Alexander polynomial is topologically slice \cite{MR804721,MR1201584}, an instrumental theorem in producing topologically but not smoothly slice knots. Apart from distinguishing topological and smooth sliceness, topologically but not smoothly slice knots also imply the existence of exotic smooth structures on $\mathbb{R}^4$ \cite[Lemma 1.1]{MR816673}. In addition, Freedman's proof implies that any slice disk $D$ with $\pi_1(B^4 - D) = \mathbb{Z}$ is unique up to topological isotopy (see for example \cite[Theorem 1.2]{conway2019enumerating}), which has been used in producing disks in $B^4$ which are topologically but not smoothly isotopic relative to their shared boundary \cite{Haydenexotic,dai2022equivariant}.

It is natural to ask for analogous results which guarantee the existence of an equivariant topological slice disk bounded by a symmetric knot. In the case of strongly invertible knots, Issa and the first author showed that there are knots with trivial Alexander polynomial which do not bound any equivariant topological slice disk \cite[Example 4.8]{BoyleIssa2}; see Question \ref{question:stronglyinvertible}. In contrast, for strongly negative amphichiral knots we show that the topological slice disk implied by the trivial Alexander polynomial condition respects the strongly negative amphichiral symmetry.

\begin{definition}
A \emph{strongly negative amphichiral knot} $(K,\rho)$ is a knot $K \subset S^3$ and an involution $\rho:S^3 \to S^3$ such that $\rho(K) = K$, and the fixed-point set of $\rho$ is two points on $K$. 
\end{definition}
An involution with exactly two fixed points in $S^3$ is conjugate to point reflection \cite{MR111044,MR155323}, so a strongly negative amphichiral knot always has a symmetric diagram as seen in Figure \ref{fig:12a_435}. 

\begin{theorem}\label{thm:alex1}
Let $(K,\rho)$ be a strongly negative amphichiral knot with Alexander polynomial $\Delta_K(t) = 1$. Then $(K,\rho)$ is equivariantly topologically slice with a slice disk $D$ such that $\pi_1(B^4\backslash D)\cong \mathbb{Z}$.
\end{theorem}

\begin{figure}
\begin{overpic}[width=.5\linewidth, grid = false]{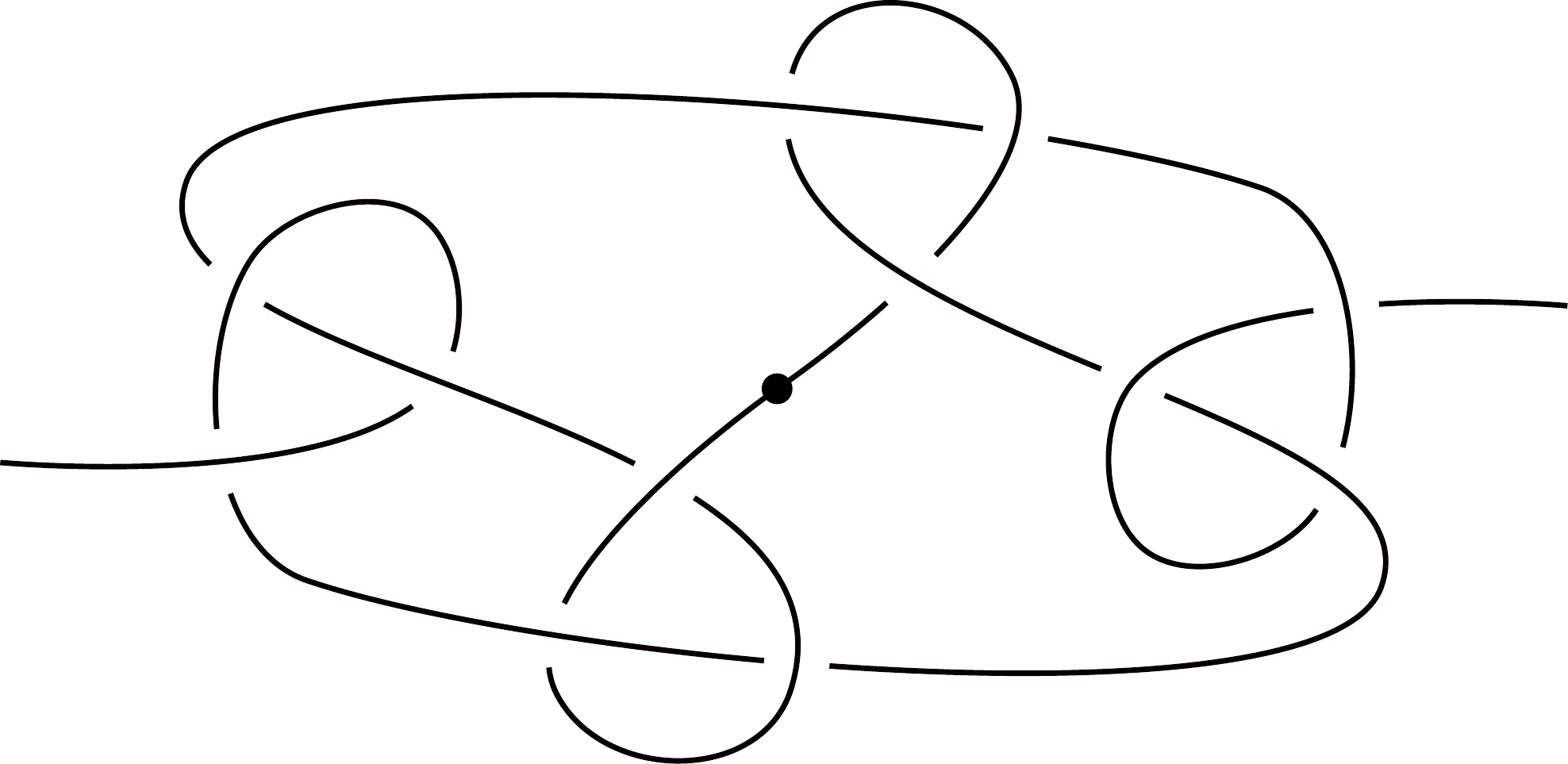}
\end{overpic} 
\caption{A strongly negative amphichiral knot. The symmetry is given by point reflection $(x,y,z) \mapsto (-x,-y,-z)$ across the indicated point $(0,0,0)$, and the two ends connect at infinity (the other fixed point).}
\label{fig:12a_435}
\end{figure}

One motivation for this theorem is the existence of the \emph{half-Alexander polynomial} \cite[Definition 4.1]{boyle2022negative}, which acts in many ways as an equivariant analog of the Alexander polynomial for strongly negative amphichiral knots. For example, it satisfies a version of the Fox-Milnor condition to obstruct equivariant sliceness \cite[Lemma 3.6]{MR433474} and it satisfies an equivariant skein relation analogous to the Conway polynomial \cite[Theorem 1.1]{boyle2022negative}. Hence one might expect that a strongly negative amphichiral knot with a trivial half-Alexander polynomial is equivariantly topologically slice. Since knots with a trivial half-Alexander polynomial have a trivial Alexander polynomial, this would imply Theorem \ref{thm:alex1}.

We know of two approaches to proving that a knot with trivial Alexander polynomial is topologically slice: Freedman's original approach using surgery theory, and the approach of ambient surgeries on Seifert surfaces used by Garoufalidis and Teichner \cite{MR2153483}. Our proof of Theorem \ref{thm:alex1} follows a simplified version of the surgery theory approach (see \cite[Theorem 1.14]{DISCEMBEDDING}) since strongly negative amphichiral knots do not admit equivariant Seifert surfaces. The main difficulty in our setting is the orientation-reversing nature of the symmetry. In particular, we must work with non-orientable 4-manifolds. For an overview of the techniques involved, see Section \ref{sec:background}.

\subsection{Examples}
We indicate some constructions of strongly negative amphichiral knots with trivial Alexander polynomial. One such knot was constructed by Van Buskirk in \cite[Example 2]{MR715764}. An infinite family of examples can also be constructed via a modification of the construction used in the proof of \cite[Theorem 6.1]{boyle2022negative}. Specifically, in \cite[Figure 13]{boyle2022negative} choose any values of $b_i$ and $c_i$, and replace the $+1$ surgery curve with its positive Whitehead double and the $-1$ surgery curve with its negative Whitehead double. The result is a surgery diagram for a strongly negative amphichiral knot with trivial half-Alexander polynomial (which can be computed as in \cite[Section 5]{boyle2022negative}), and hence trivial Alexander polynomial. See Figure \ref{fig:example1} for the knot produced from this construction when all parameters are $0$ except $c_1 = 1$. Examples can also be produced through equivariant satellite operations, provided that both the companion and pattern knots have trivial Alexander polynomials.

\begin{figure}
\begin{overpic}[width=.5\linewidth, grid = false]{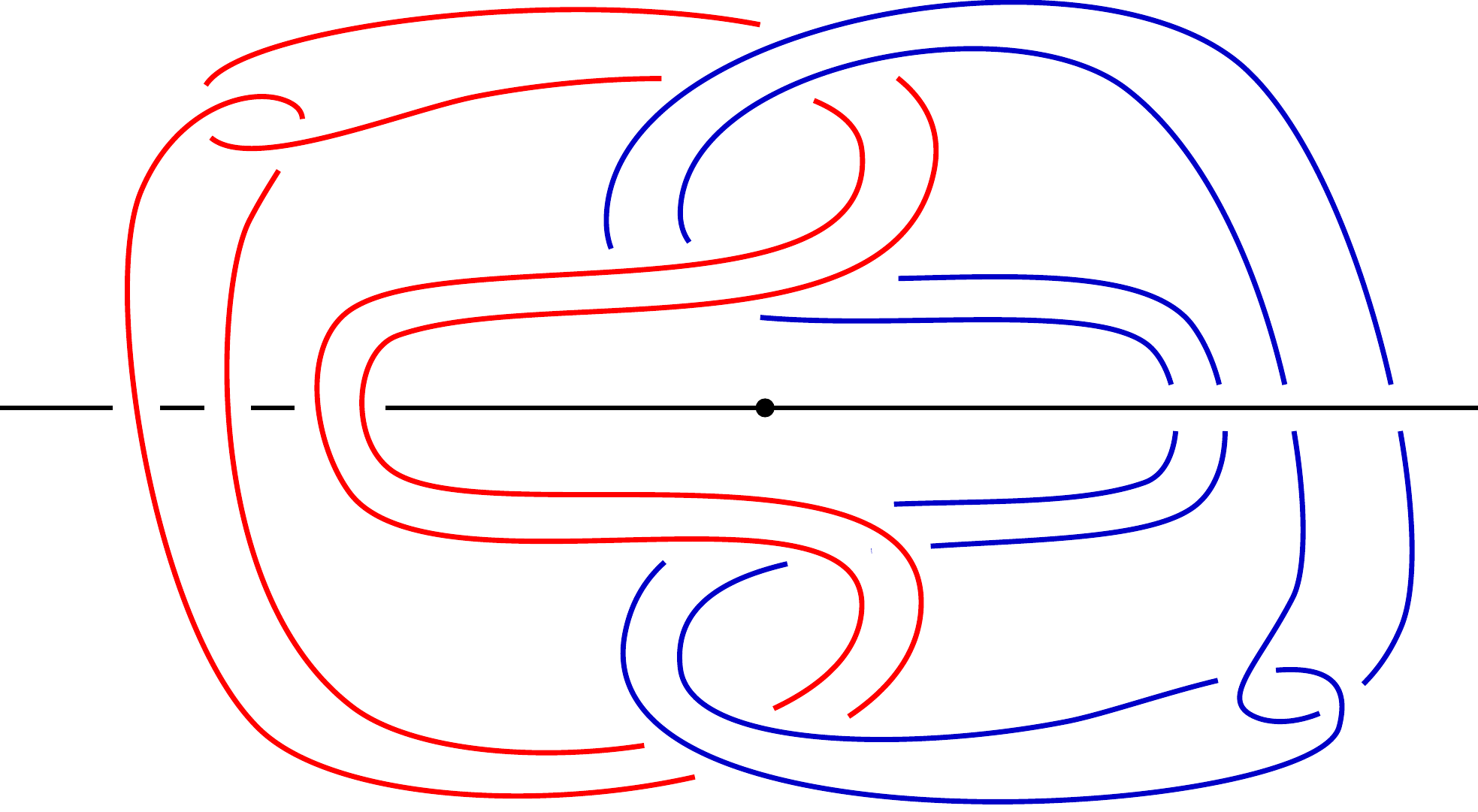}
\put (8, 4) {$+1$}
\put (90,45) {$-1$}
\put (-5,26) {$K$}
\end{overpic} 
\begin{overpic}[width=.4\linewidth, grid = false]{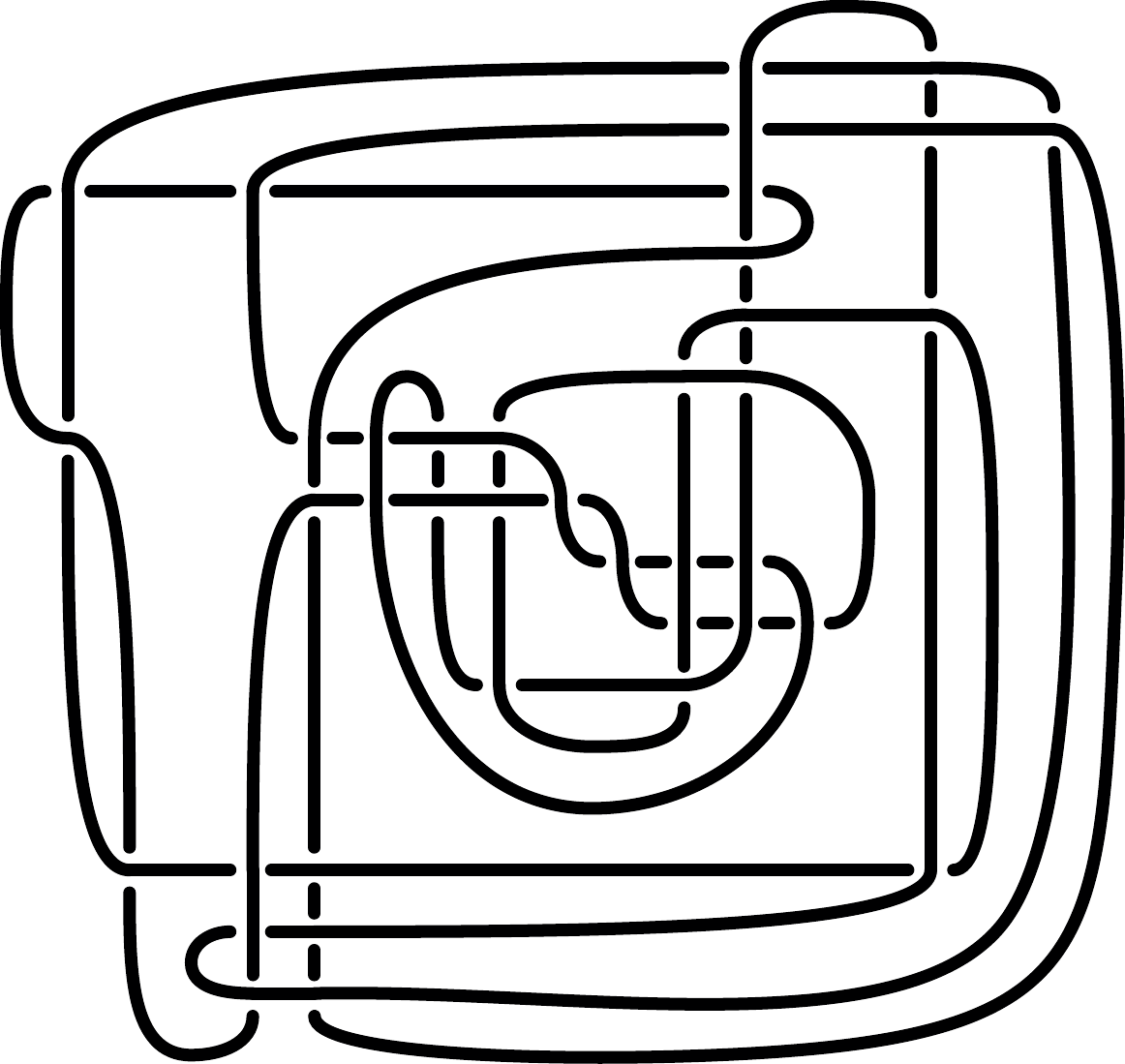}

\end{overpic}

\caption{A strongly negative amphichiral knot $K$ with trivial Alexander polynomial. The diagram on the left is an equivariant surgery diagram for $K$. Blowing down the surgery curves produces a knot diagram for $K$ which we simplified using KLO \cite{KLO} and SnapPy \cite{SnapPy} to produce the diagram on the right.}
\label{fig:example1}
\end{figure}

\subsection{Open questions}
We conclude with some open questions related to equivariant slice disks.

\begin{question}
If $\overline{\rho}: B^4 \to B^4$ is a locally linear order 2 homeomorphism which restricts to a point reflection $\rho$ on $S^3 = \partial B^4$, then is $\overline{\rho}$ conjugate to the standard involution (the cone of $\rho$) in the homeomorphism group of $B^4$?
\end{question}
Implicit in Theorem \ref{thm:alex1} is the construction of many such involutions on $B^4$, so that proving that a strongly negative amphichiral knot with trivial Alexander polynomial is not equivariantly topologically slice with respect to the standard involution would answer this question. On the other hand, all known equivariant smooth slice disks are invariant under the standard involution on $B^4$ (see \cite{BoyleIssa} for some examples).

\begin{question} \label{question:smoothlyslice}
Is there a strongly negative amphichiral knot with trivial Alexander polynomial which is not smoothly equivariantly slice?
\end{question}
Answering this question in the positive would immediately produce a strongly negative amphichiral knot which is topologically but not smoothly equivariantly slice by Theorem \ref{thm:alex1}. In fact, Issa and the first author have produced an obstruction to smooth equivariant slice disks using Donaldson's theorem \cite[Theorem 1.3]{BoyleIssa}. However, this theorem does not apply to any known examples of knots with trivial Alexander polynomial. Question \ref{question:smoothlyslice} is closely related to \cite[Problem 21]{FT}, which asks for an amphichiral knot with a trivial Alexander polynomial which is not smoothly slice.

One may hope to answer Question \ref{question:smoothlyslice} by developing an equivariant analog of knot Floer homology, which may also be helpful for studying the following question (at least in the smooth category).

\begin{question} \label{question:isotopicdisks}
Are there equivariant slice disks for some strongly negative amphichiral knot which are isotopic but not equivariantly isotopic?
\end{question}
We suspect that there are such disks; in the proof of Theorem \ref{thm:alex1} (specifically Lemma \ref{lemma:equivariantmeridian}) there are two choices of $\rho$-invariant meridian to which we may equivariantly attach a handle to construct an equivariant slice disk. These two slice disks are isotopic (since the meridians are isotopic), but we suspect they are not equivariantly isotopic. 

In \cite{MR3523068}, Feller showed that the degree of the Alexander polynomial is an upper bound for the topological slice genus, and it would be interesting to see if there is an equivariant version of this theorem.

\begin{question}
Is the degree of the Alexander polynomial an upper bound on the equivariant topological slice genus for strongly negative amphichiral knots?
\end{question}

Feller's proof relies on a modification of a Seifert surface, and is closely related to the ambient surgery technique used in \cite{MR2153483}. In the setting of strongly negative amphichiral knots, we do not have equivariant Seifert surfaces, and hence Feller's proof cannot be directly applied.

Finally, as indicated by \cite[Example 4.8]{BoyleIssa2}, any statement for strongly invertible knots analogous to Theorem \ref{thm:alex1} must involve more than just the Alexander polynomial. 

\begin{question} \label{question:stronglyinvertible}
Is there a homological condition which guarantees that a strongly invertible knot is equivariantly topologically slice?
\end{question}

Question \ref{question:stronglyinvertible} is closely related to a question about slice links, as we can see from the following construction. Extending a strongly invertible symmetry to the $0$-surgery $S^3_0(K)$ on $K$ produces a fixed-point set consisting of a 2-component link, and the quotient of $S^3_0(K)$ is a homology $S^3$ containing the image $L$ of this 2-component link. Furthermore, $K$ is equivariantly slice if $L$ is freely slice. Indeed, the double branched cover over a pair of freely slice disks for $L$ is a homotopy $S^1$ with boundary $S^3_0(K)$, which is an equivariant slice disk complement for $K$.

\subsection{Organization}
In Section \ref{sec:background} we give background for equivariant slice disks, surgery theory, and $Spin$ and $Pin^-$ structures. In Section \ref{sec:proof} we prove Theorem \ref{thm:alex1}.

\subsection{Acknowledgments}
We would like to thank Liam Watson for some helpful comments, and the authors of \cite{DISCEMBEDDING} for an excellent introduction to Freedman's work. The second author is partially supported by the Pacific Institute for the Mathematical Sciences (PIMS). The research and findings may not reflect those of the institute.

\section{Background} \label{sec:background}
\subsection{Equivariantly slice knots}
We begin by introducing the definition of an equivariant slice disk for a strongly negative amphichiral knot. 

\begin{definition}
A strongly negative amphichiral knot $(K,\rho)$ is \emph{equivariantly (topologically) slice} if there is a locally linear order 2 homeomorphism $\rho'\colon B^4 \to B^4$ and a $\rho'$-invariant locally flat and properly embedded disk $D \subset B^4$ with $\partial D = K$ such that $\rho'$ restricts to $\rho$ on $\partial B^4 = S^3$. Here $D$ is called an \emph{equivariant topological slice disk for $K$}.
\end{definition}

Here we require the involution $\rho'$ to be locally linear in order to avoid pathological fixed-point sets, as is standard for studying group actions on topological manifolds. For convenience, we recall the definition below.

\begin{definition}
An order 2 homeomorphism $\rho$ on an $n$-dimensional topological manifold is \emph{locally linear} if each fixed point $x$ of $\rho$ has a $\rho$-invariant neighborhood which is equivariantly homeomorphic to $\mathbb{R}^n$ with a linear $\mathbb{Z}/2\mathbb{Z}$ action.
\end{definition}

\subsection{A brief overview of surgery theory}

The main technical input in the proof of Theorem \ref{thm:alex1} is surgery theory, which we use to construct the equivariant slice disk complement. Surgery theory provides a method to modify a given manifold $M$, by performing surgery on spheres representing specific homotopy classes in $M$, in order to obtain a new manifold with a prescribed homotopy type. In our case, we will modify a 4-manifold with a prescribed boundary to obtain a homotopy $S^1$. We will perform surgeries on 1- and 2-dimensional spheres to achieve $\pi_1(M) \cong \mathbb{Z}$ and $\pi_2(M) \cong 0$, which is sufficient to guarantee that $M$ is a homotopy $S^1$ (see for example, \cite[Proposition 10.18]{MR2061749}).

In order to perform surgery on a sphere, one necessary condition is that the sphere must be framed; that is the normal bundle is trivial. A standard way of guaranteeing that the spheres are framed is to work with manifolds with extra bundle data, usually packaged in a notion called normal maps. For simplicity, we will avoid normal maps and instead obtain framed spheres by taking advantage of working in low-dimensions: a 1-sphere is framed if and only if its normal bundle is orientable, and a 2-sphere is homotopic to a framed sphere if and only if its self-intersection number is even. Even self-intersection numbers can be guaranteed provided the 4-manifold admits a $Pin^-$ structure. See Section \ref{sec:pin} for a discussion on $Pin^-$ structures.   

In order to perform surgery on a sphere, the sphere must also be embedded. For 1-spheres in a 4-manifold, this is automatic since every 1-sphere is homotopic to an embedded one by Whitney's embedding theorem. 

The situation for 2-spheres is more complicated. Even if a 2-sphere is embedded, surgery on this sphere may alter $\pi_1(M)$ since the meridian of the 2-sphere might be nontrivial in the fundamental group of the new manifold. To avoid this, it is desirable to ask for a geometrically dual sphere $S'$ to the sphere $S$ on which the surgery is performed. Here, $S'$ (not necessarily embedded) and $S$ intersect geometrically once and the presence of $S'$ provides a null-homotopy of the meridian of $S$ in the manifold obtained from surgery. Hence we would like to find a basis of $\pi_2(M)$ such that half of the basis is represented by disjoint, framed, and embedded 2-spheres, each of which admits a geometric dual sphere. Here we require the geometric dual spheres to form the other half of the basis. Then surgery on these embedded spheres will annihilate $\pi_2(M)$. 

The algebraic obstruction to the existence of such a basis is known as the surgery obstruction; this is the Witt class of a certain quadratic form on $\pi_2(M)$, which is an element of Wall's obstruction group $L_4(\mathbb{Z}[\pi_1(M)]^w)$. (We refer the reader to \cite[Definition 4.22 and Definition 4.28]{MR1937016} or \cite{MR1687388} for the definitions.) This group only depends on $\pi_1(M)$ and the orientation character $w\colon\pi_1(M) \to \mathbb{Z}/2\mathbb{Z}$. If the surgery obstruction vanishes, and the fundamental group is good, then Freedman's sphere embedding theorem produces a collection of spheres on which one can perform surgeries to annihilate $\pi_2(M)$ without modifying $\pi_1(M)$ (see for example \cite[Theorem 5.1A]{MR1201584} or \cite[Chapter 20.3]{DISCEMBEDDING}).

\subsection{\texorpdfstring{$Spin$}{Spin} and \texorpdfstring{$Pin^-$}{Pin} structures} \label{sec:pin}

Recall that a compact oriented 4-manifold $M$ admits a $Spin$ structure if and only if the intersection form is even. We will be considering non-orientable manifolds, which cannot be $Spin$. Instead we use a $Pin^-$ structure, which like a $Spin$ structure guarantees that the self-intersection pairing of elements in $H_2(M;\mathbb{Z}/2\mathbb{Z})$ vanishes. We briefly recall some basic definitions and properties below; see \cite{MR1171915} for details.

A $Spin$ structure on an $SO(n)$ bundle over a manifold is a lift of the $SO(n)$ bundle to a $Spin(n)$ bundle, where $Spin(n)$ is the (degree 2) universal cover of $SO(n)$ for $n \geq 3$. When $n < 3$, a $Spin$-structure on an $SO(n)$ bundle is a $Spin$ structure on some stabilization of the bundle. Indeed a $Spin$ structure on a vector bundle is a stable structure; the set of $Spin$ structures on an $SO(n)$ bundle is in bijective correspondence with the set of $Spin$ structures on each stabilization of the bundle. A $Spin$ structure on an oriented manifold is a $Spin$ structure on its orthonormal frame bundle (the choice of Riemannian metric does not change the isomorphism class of the $SO(n)$ bundle). An oriented manifold $M$ admits a $Spin$ structure if and only if the second Stiefel-Whitney class $w_2(M)$ vanishes, in which case the set of $Spin$ structures has a non-canonical bijection to $H^1(M;\mathbb{Z}/2\mathbb{Z})$.

A $Pin^-$ structure on a (not necessarily orientable) compact manifold $M$ is equivalent to a $Spin$ structure on $TM \oplus \det(TM)$ for an arbitrary orientation on $TM \oplus \det(TM)$. When $M$ is orientable, $\det(TM)$ is trivial so that a $Pin^-$ structure on $M$ is equivalent to a $Spin$-structure on $M$. A manifold $M$ admits a $Pin^-$ structure if and only if $w_2(M) + w_1(M)^2 = 0$ (where $w_i(M)$ is the $i$-th Stiefel-Whitney class), in which case the set of $Pin^-$ structures has a non-canonical bijection to $H^1(M;\mathbb{Z}/2\mathbb{Z})$. 

Next we discuss the extension of a $Pin^-$ structure from a 4-manifold $M$ to a 4-manifold $M'$ obtained from $M$ by surgery along framed 1-spheres. Give a framed embedded 1-sphere $S$ in $M$, there are two choices of framing on $S$ (since $\pi_1(SO(3)) \cong \mathbb{Z}/2\mathbb{Z}$). For one choice of framing, the $Pin^-$ structure on $M$ extends to $M'$, and for the other choice of framing it does not. Indeed, there is a unique $Pin^-$ structure on the 2-handle $h = D^2 \times D^3$ which we attach to $M \times I$, and by attaching $h$ so that the $Pin^-$ structure on $S$ agrees with the $Pin^-$ structure on $h$ we will get a $Pin^-$ structure on the cobordism from $M$ to $M'$ which we can restrict to $M'$.

Finally, we will make use of $Spin$ and $Pin^-$ bordism groups, which we define here for convenience. See \cite{MR0385836} or \cite{MR1171915} for more details.
\begin{definition}
Let $X$ be a topological space. Then the $i$-th $Pin^-$ bordism group $\Omega_i^{Pin^-}(X)$ of $X$ consists of bordism classes of triples $[M,\xi,f]$, where $M$ is an $i$-dimensional manifold with a $Pin^-$ structure $\xi$ and a continuous map $f:M \to X$. Here $(M,\xi,f)$ is equivalent to $(N,\psi,g)$ if there is a $Pin^-$ cobordism $W:M \to N$ and a map $F:W \to X$ such that the $Pin^-$ structure and map $F$ on $W$ restrict to $(\xi,f)$ on $M$ and $(\psi,g)$ on $N$. The group operation is induced by disjoint union.
\end{definition}
The $Spin$ bordism groups $\Omega_i^{Spin}(X)$ are defined similarly. Note that both $\Omega_i^{Pin^-}(X)$ and $\Omega_i^{Spin}(X)$ are generalized homology theories.

\section{Proof of Theorem \ref{thm:alex1}}\label{sec:proof}
In this section we prove Theorem \ref{thm:alex1}. Before we begin the proof, we state some necessary lemmas. To begin, Lemmas \ref{lemma:0involution} and \ref{lemma:pinbordismclass} provide a 4-manifold on which we will perform some surgeries. 
\begin{lemma}\cite[Proof of Lemma 2.3]{MR2554510} \label{lemma:0involution}
The strongly negative amphichiral symmetry $\rho$ induces a free involution on the $0$-surgery $S^3_0(K)$. Furthermore, the quotient $Q = S^3_0(K)/\rho$ has $H^1(Q;\mathbb{Z}) = \mathbb{Z}$.
\end{lemma}

Note that $Q$ has two $Pin^-$ structures, since $H^1(Q,\mathbb{Z}/2\mathbb{Z}) = \mathbb{Z}/2\mathbb{Z}$. We will arbitrarily choose one such structure $\xi$ to work with, and let $f:Q \to S^1$ be a map corresponding to a generator of $H^1(Q;\mathbb{Z})$. Note that $[Q,\xi,f]$ is an element of the $Pin^-$ bordism group $\Omega^{Pin^-}_3(S^1)$. 

\begin{lemma} \label{lemma:pinbordismclass}
The bordism class $[Q,\xi,f]$ is trivial in $\Omega^{Pin^-}_3(S^1)$ if and only if Arf$(K) = 0$. Moreover, the null-cobordism $(V_0,f:V_0 \to S^1)$ can be chosen so that for each simple closed curve $\gamma \subset V_0$, if $f_*([\gamma]) = 0 \in H_1(S^1;\mathbb{Z})$ then the normal bundle of $\gamma$ is orientable.
\end{lemma}
\begin{proof}
First, observe that there is an isomorphism $\Theta\colon\Omega_3^{Pin^-}(S^1) \to \Omega_2^{Pin^-}(*)$ by applying the Meier-Vietoris sequence for generalized homology theories (see for example \cite[Theorem 7.19]{MR0385836}), and using the fact that $\Omega_3^{Pin^-}(*) = 0$ (see for example \cite[Theorem 5.1]{MR1171915}). Let $F = f^{-1}(*)$ for some regular value $*$ of $f$. Then the image $\Theta([Q,\xi,f])$ is represented by $F$ together with the $Pin^-$ structure $\xi|_F$ it inherits from $Q$. The Brown-Arf invariant $\beta$ classifies $\Omega_2^{Pin^-}(*) \cong \mathbb{Z}/8\mathbb{Z}$ (see for example \cite[Lemma 3.6]{MR1171915}), so we need to show that $\beta([F,\xi|_F])$ vanishes if and only if Arf$(K)$ vanishes.

In fact $F$ is orientable, and hence $\xi|_F$ lifts to a $Spin$ structure $\overline{\xi}$ so that $[F,\overline{\xi}]$ is a class in $\Omega_2^{Spin}(*) \cong \mathbb{Z}/2\mathbb{Z}$, which is classified by the Arf invariant Arf$([F,\overline{\xi}])$. Furthermore, by \cite[Proposition 3.8]{MR1171915}, there is a commutative diagram
\[
\begin{tikzcd}
\Omega_2^{Spin}(*) \arrow[hookrightarrow]{r} \arrow{d}{Arf}& \Omega_2^{Pin^-}(*) \arrow{d}{\beta} \\
\mathbb{Z}/2\mathbb{Z} \arrow[hookrightarrow]{r} & \mathbb{Z}/8\mathbb{Z}
\end{tikzcd}
\]
where the map $\mathbb{Z}/2\mathbb{Z} \to \mathbb{Z}/8\mathbb{Z}$ is the unique injective homomorphism. Hence $\beta([F,\xi|_F]) = 0$ if and only if Arf$([F,\overline{\xi}]) = 0$. Now $F$ lifts to a pair of orientable surfaces $F_0$ and $F_1$ in $S^3_0(K)$, and $[F,\overline{\xi}] \cong [F_0,\overline{\eta}]$, where $\overline{\eta}$ is the restriction of the $Spin$ structure on $S^3_0(K)$ to $F_0$. Hence Arf$([F,\overline{\xi}]) = $ Arf$([F_0,\overline{\eta}]) = $ Arf$(K)$.

For the second statement we will construct the null-cobordism with the desired property directly. First, note that by an appropriate isotopy of $f$, we may assume that $F = f^{-1}(*)$ is connected. Since Arf$([F,\overline{\xi}]) = 0$, there is a $Spin$ handlebody $H$ with $\partial H = F$. Now take $Q \times [0,1]$ and glue $H \times [-\varepsilon,\varepsilon]$ to $Q \times \{1\}$ in a neighborhood $U$ of $F \times \{1\}$. Indeed, $*$ is a regular value so that $U \cong F \times [-\varepsilon,\varepsilon]$, and hence we can identify $\partial H \times \{x\}$ with $F \times \{x\}$ for all $x \in [-\varepsilon,\varepsilon]$. Call the resulting cobordism $W_0$, and let $M = \partial W_0 - Q$. Now since $F$ is dual to $w_1(Q)$, $(Q \times\{1\}) - U$ is orientable (see for example \cite[Lemma 2.2]{MR1171915}). Hence $M$ is orientable, since it is obtained from $Q - U$ by gluing on two 3-dimensional handlebodies. In particular, the $Pin^-$ structure on $M$ lifts to a $Spin$ structure, and $M$ then bounds a $Spin$ 4-manifold $W_1$ since $\Omega_3^{Spin}(*) = 0$. In fact, we can take $W_1$ to be simply connected \cite{MR539917}.

We can now define $V_0 = W_0 \cup_M W_1$, and we can extend $f$ from $Q$ to $V_0$ as follows. First, extend $f$ to $Q \times I$ by $f(x,y) = (f(x), y)$. Identifying $U = F \times [-\varepsilon,\varepsilon]$, we may assume (up to isotopy) that $f|_U$ is projection to the second coordinate, where we identify a neighborhood of the regular value $*$ with $[-\varepsilon,\varepsilon]$. We can then extend $f$ to $W_0$ by defining $f|_{H \times [-\varepsilon,\varepsilon]}$ as projection to the second coordinate. With this extension of $f$, we can see that $f(M)$ is contractible in $S^1$ since it does not contain $*$. Hence after an isotopy of $f$ in a collar neighborhood of $M$ in $W_0$ we can assume $f(M)$ is a single point. We can then extend $f$ to $V_0$ by taking $f(W_1)$ to be the same point. Now let $\gamma\colon S^1 \to V_0$ with $[f \circ \gamma] = 0 \in H_1(S^1;\mathbb{Z})$. Note that $\gamma$ is homologous to a curve $\gamma'$ in $Q$ by the construction of $V_0$, and hence $w_1(V_0)[\gamma] = w_1(Q)[\gamma']$. Since $[f\circ\gamma']$ is even in $H_1(S^1)$, $w_1(Q)[\gamma'] = 0$ and hence the normal bundle of $\gamma$ is orientable. 
 
\end{proof}

In the course of performing surgeries we will need to ensure that the surgery obstruction vanishes. In our case, this surgery obstruction lies in the $L$-group $L_4(\mathbb{Z}[\mathbb{Z}^-])$, where $\mathbb{Z}^-$ refers to the integers together with the non-trivial homomorphism $\mathbb{Z} \to \mathbb{Z}/2\mathbb{Z}$.

\begin{lemma}\cite[Section 13A]{MR1687388}\label{lemma:L-group}
The L-group $L_4(\mathbb{Z}[\mathbb{Z}^-])$ is isomorphic to $\mathbb{Z}/2\mathbb{Z}$, and is generated by the $E_8$ form.
\end{lemma}
\begin{proof}
Let $\mathbb{Z}[1]$ indicate the group ring of the trivial group. By \cite[Theorem 12.6]{MR1687388}, there is an exact sequence
\[
L_4(\mathbb{Z}[1]) \overset{\cdot 2}{\longrightarrow} L_4(\mathbb{Z}[1]) \to L_4(\mathbb{Z}[\mathbb{Z}^-]) \to 0,
\]
where, $L_4(\mathbb{Z}[1]) \cong \mathbb{Z}$ generated by the $E_8$ form \cite[Theorem 13.A.1]{MR1687388}, and the first map is multiplication by 2. Hence $L_4(\mathbb{Z}[\mathbb{Z}^-]) \cong \mathbb{Z}/2\mathbb{Z}$, and is also generated by the $E_8$ form.
\end{proof}

The following lemma provides an equivariant attaching curve allowing us to extend a symmetry along a 2-handle attachment, which we use to construct an equivariant slice disk. 

\begin{lemma}\label{lemma:equivariantmeridian}
Let $(K,\rho)$ be a strongly negative amphichiral knot. Then there is a $\rho$-invariant meridian $\mu$ of $K$ with an $S^1[\theta] \times D^2[r,\phi]$ neighborhood where $\rho$ acts by 
\[
(\theta,r,\phi) \mapsto (\theta+\pi,r,-\phi), \mbox{ with } \theta,\phi \in [0,2\pi) \mbox{ and } r \in [0,1].
\]
Furthermore, this parametrization of the neighborhood of $\mu$ gives the Seifert framing in $S^3$. 
\end{lemma}
\begin{proof}
Let $x$ be a fixed point of $\rho$ on $K$. Then there is a neighborhood of $x$ which is homeomorphic to $\mathbb{R}^3$ with the action $(x_1,x_2,x_3) \mapsto (-x_1,-x_2,-x_3)$, where $K$ is the $x_1$-axis and $x$ is the origin (see \cite{MR111044} and \cite{MR155323}). We can then take 
\[
\mu = \{x_1 = 0, x_2^2+x_3^2 = 1\},
\]
which has an $S^1 \times D^2$ neighborhood with the desired $\rho$ action under an appropriate parametrization.
\end{proof}
Apart from guaranteeing that the Arf invariant vanishes (see Lemma \ref{lemma:pinbordismclass}), the vanishing Alexander polynomial is also used in the following lemma, which will guarantee that a certain quadratic form is non-degenerate.
\begin{lemma} \label{lemma:Qhomologyvanishes}
If $\Delta_K(t) = 1$, then $H_1(\widetilde{Q};\mathbb{Z}) \cong 0$, where $\widetilde{Q}$ is the infinite cyclic cover of $Q$.
\end{lemma}
\begin{proof}
By \cite{MR433474}, $\widetilde{Q}$ is homeomorphic to the infinite cyclic cover $\widetilde{S^3_0}(K)$ of $S^3_0(K)$, and by \cite{MR167976},
\[
H_1(\widetilde{S^3_0}(K);\mathbb{Z}) = 0
\]
when $\Delta_K(t) = 1$.
\end{proof}

We can now prove Theorem \ref{thm:alex1}.

\begin{proof}[Proof of Theorem \ref{thm:alex1}]
We will build a compact 4-manifold $W$ with boundary $S^3_0(K)$, along with an involution on $W$ which restricts to the involution on $S^3_0(K)$ induced by $\rho$ (see Lemma \ref{lemma:0involution}). Furthermore, we will guarantee that $W$ is a homotopy $S^1$ and that $H_1(W;\mathbb{Z})$ is generated by a meridian $\mu$ of $K$. Once we have constructed $W$, gluing a 2-handle $h = D^2 \times D^2$ to $W$ along $\mu$ then produces a topological 4-ball in which the co-core of $h$ is a slice disk for $K$. Moreover, $h$ can be attached equivariantly with respect to $\rho$ so that $K$ is equivariantly topologically slice. Indeed, we attach $h$ to the $\rho$-invariant neighborhood of $\mu$ as described in Lemma \ref{lemma:equivariantmeridian}, and we equip $h = D^2[s,\theta] \times D^2[r,\phi]$ with the involution given by $\rho(s,\theta,r,\phi) = (s,\theta+\pi,r,-\phi)$, which restricts to the involution in Lemma \ref{lemma:equivariantmeridian} as desired. (Observe that the fixed-point set is the diameter $(0,0,r,0) \cup (0,0,r,\pi)$ of the co-core of $h$.)

To build $W$, we first consider $Q = S^3_0(K)/\rho$. We will show that $Q$ bounds a (non-orientable) 4-manifold $V$, which is a homotopy $S^1$ and for which $H_1(V;\mathbb{Z})$ is generated by a circle $\gamma$ which lifts to a meridian $\mu$ of $K$. In fact, $W$ is the orientation cover of $V$ and hence also a homotopy $S^1$, and the deck transformation involution is the desired involution.

The rest of the proof consists of the construction of $V$. Since $\Delta_K(t) = 1$, we have $Arf(K) = 0$ so that Lemma \ref{lemma:pinbordismclass} produces a $Pin^-$ 4-manifold $V_0$ with boundary $Q$ and a map $f_0\colon V_0 \to S^1$ which restricts to $f\colon Q \to S^1$, where the homotopy class of $f$ corresponds to a generator of $H^1(Q;\mathbb{Z})$. First, we will perform surgeries in the interior of $V_0$ to obtain a 4-manifold $V_1$ so that the following diagram commutes and $f_1$ induces an isomorphism $\pi_1(V_1) \to \mathbb{Z}$. 

\[ \begin{tikzcd}
Q \arrow{r}{\iota} \arrow{dr}{f}& V_1 \arrow{d}{f_1} \\
 & S^1
\end{tikzcd}
\]
To do so, consider a finite presentation for the kernel of $(f_0)_*\colon \pi_1(V_0) \to \pi_1(S^1)$. By Lemma \ref{lemma:pinbordismclass}, the generators of ker$((f_0)_*)$ can be represented by framed (and embedded) circles since they have orientable normal bundles. Performing surgery on each of these circles then produces $V_1$. Here we choose the framings so that $V_1$ inherits a $Pin^-$ structure from $V_0$. 

The existence of this $Pin^-$ structure on $V_1$ implies that $w_2(V_1) + w_1^2(V_1) = 0$. Since $w_2(V_1) + w_1^2(V_1)$ is the second Wu class, we have that for any immersed sphere $S$ representing an element of $\pi_2(V_1)$, the self-intersection pairing $\langle [S] \cup [S],[V_1,Q] \rangle \equiv 0$ (mod $2$). By introducing local kinks, each of which changes the Euler number by $\pm 2$, we may further choose a representative immersed sphere $S$ with trivial normal bundle for each class in $\pi_2(V_1)$. Here $S$ is unique up to regular homotopy \cite[Theorem 8.3]{MR119214}. We can then use these Euler number 0 representatives to define an intersection pairing $\lambda\colon\pi_2(V_1) \times \pi_2(V_1) \to \mathbb{Z}[\pi_1(V_1)]$ as in \cite[Definition 11.2]{DISCEMBEDDING} (with arbitrarily chosen whiskers).

We will show that this intersection pairing $\lambda$ is non-degenerate so that $\lambda$, together with its quadratic enhancement $s$ given by the self-intersection number, represents an element $(\pi_2(V_1),\lambda,s)$ in the surgery obstruction group $L_4(\mathbb{Z}[\pi_1(V_1)]^w)$, where $w$ is the orientation character. To see that $\lambda$ is non-degenerate, note that $\lambda$ is isomorphic to the $\mathbb{Z}[\pi_1(V_1)]$-equivariant intersection form 
\[
\lambda\colon H_2(V_1;\mathbb{Z}[\pi_1(V_1)]) \times H_2(V_1;\mathbb{Z}[\pi_1(V_1)]^w) \to \mathbb{Z}[\pi_1(V_1)],
\]
and that we have the following exact sequence:
\begin{equation}\label{eqn:V1Qexact}
\dots \to H_2(V_1;\mathbb{Z}[\pi_1(V_1)]^w) \overset{p}{\to} H_2(V_1,Q;\mathbb{Z}[\pi_1(V_1)]^w) \to H_1(Q;\mathbb{Z}[\pi_1(V_1)]^w) \to \cdots
\end{equation}
Noting that $H_1(Q;\mathbb{Z}[\pi_1(V_1)]^w) \cong H_1(\widetilde{Q};\mathbb{Z})$ as abelian groups, Lemma \ref{lemma:Qhomologyvanishes} gives an isomorphism 
\[
H_1(Q;\mathbb{Z}[\pi_1(V_1)]^w) \cong 0,
\]
so that $p$ in Equation (\ref{eqn:V1Qexact}) is surjective. Now for any class $\alpha$ in $H_2(V_1;\mathbb{Z}[\pi_1(V_1)]^w)$, Poincar\'{e} duality produces a class in $H_2(V_1,Q;\mathbb{Z}[\pi_1(V_1)]^w)$ which pairs non-trivially with $\alpha$. Hence $\lambda$ is non-degenerate and $(\pi_2(V_1),\lambda,s) \in L_4(\mathbb{Z}[\pi_1(V_1)]^w)$.

Now observe that $\pi_1(V_1) \cong \mathbb{Z}$ and the orientation character $w: \pi_1(V_1) \to \mathbb{Z}/2\mathbb{Z}$ is surjective since $V_1$ is non-orientable. We may assume that the surgery obstruction $(\pi_2(V_1),\lambda,s)$ is trivial in $L_4(\mathbb{Z}[\pi_1(V_1)]^w) \cong L_4(\mathbb{Z}[\mathbb{Z}^-])$; if it is not, then by Lemma \ref{lemma:L-group} we can replace $V_1$ with its connected sum with the $E_8$ manifold $V_1 \# E_8$ to obtain a vanishing surgery obstruction.

Now since $(\pi_2(V_1),\lambda,s)$ is trivial in $L_4(\mathbb{Z}[\pi_1(V_1)]^w)$, we have a stabilization 
\[
V_2 = V_1 \# (S^2 \times S^2)^k
\] 
for which $(\pi_2(V_2),\lambda,s)$ is hyperbolic. Along with the fact that $\mathbb{Z}$ is a good group, we can then apply the sphere embedding theorem (see for example \cite[Theorem 5.1A]{MR1201584} or \cite[Chapter 20.3]{DISCEMBEDDING}). The result is that $\pi_2(V_2)$ has a basis of framed immersed spheres
\[
\{S_1,\dots,S_n,S_1',\dots,S_n'\},
\]
where the spheres $S_1,\dots,S_n$ are embedded and pairwise disjoint, and for each $i$ the sphere $S_i'$ is geometrically dual to $S_i$. We can then perform surgery on $S_1,\dots,S_n$ to obtain a 4-manifold $V$. Note that $V$ is a homotopy $S^1$, and the generator of $H_1(V;\mathbb{Z})$ lifts to a meridian of $K$ by construction.

\end{proof}

\bibliography{bibliography}
\bibliographystyle{alpha}
\end{document}